\documentclass{amsart}

 \newtheorem{Theorem}{Theorem}[section]
 
 \newtheorem{lemma}[Theorem]{Lemma}
 \newtheorem{proposition}[Theorem]{Proposition}
 \theoremstyle{definition}
 
 \theoremstyle{remark}
 \newtheorem{rem}[Theorem]{Remark}
  \theoremstyle{remark}

 \numberwithin{equation}{section}
\usepackage{enumitem}
\usepackage{hyperref}

\newcommand{\R}{\mathbb{R}}

\newcommand{\N}{\mathbb{N}}
\newcommand{\Z}{\mathbb{Z}}
\newcommand{\CosineC}{\left(C(t)\right)_{t\in\mathbb R}}
\newcommand{\cosinec}{\left(c(t)\right)_{t\in\mathbb R}}
\begin{document}

\title[A $0-2$ law for $\limsup$ to $\infty$]
 {A $0-2$ law for cosine families with $\limsup$ to $\infty$}

%
%
\author[Schwenninger]{Felix L. Schwenninger}
\address{Department of Applied Mathematics, \\ University of Twente, P.O. Box 217, \\ 7500 AE Enschede, The Netherlands}
\email{f.l.schwenninger@utwente.nl}
\thanks{The first named author has been supported by the Netherlands Organisation for Scientific Research (NWO), grant no. 613.001.004.}
\author[Zwart]{Hans Zwart}
\address{Department of Applied Mathematics, \\ University of Twente, P.O. Box 217, \\ 7500 AE Enschede, The Netherlands}
\email{h.j.zwart@utwente.nl}

\date{\today}

\begin{abstract} 
For $\left(C(t)\right)_{t\in\mathbb R}$ being a cosine
family on a unital normed algebra, we show that the estimate $\limsup_{t\to
  \infty^{+}}\|C(t) - I\| <2$ implies that $C(t)=I$ for all $t\in\mathbb R$.
This generalizes the result that $\sup_{t\geq0}\|C(t)-I\|<2$ yields that $C(t)=I$ for all
$t\geq0$. We also state the corresponding result for discrete cosine families and for semigroups. 
\end{abstract}

\maketitle

\section{Introduction}
In the recent past, laws of the form 
\begin{align}
\limsup_{t\to0} \|C(t)-I \| < r &\implies \lim_{t\to0} \|C(t)-I\|=0,\tag{{\em limsup-law}}\label{limsup}\\
\sup_{t\in\mathbb{ R}}\|C(t)-I\| <r &\implies C(t)=I \ \forall t\in\mathbb{R}, \tag{{\em sup-law}}\label{sup}
\end{align}
where $r>0$ and $\left(C(t)\right)_{t\in\R}$ is a cosine family of elements in a unital Banach algebra $A$ (with identity element $I$) were studied, see \cite{UlmerSeminare2012,BobrowskiZeroTwo,ChojnackiZeroTwo,EsterleZeroTwo} and \cite{Fackler2013,SchweZwaZeroTwo} for the special case where $\CosineC$ is strongly continuous and $A=\mathcal{B}(X)$ is the Banach algebra of bounded operators on a Banach space $X$. 
For both, the \ref{limsup} and the \ref{sup} the largest possible constant $r$ was shown to be $2$.

In this note we consider the condition
\begin{equation}\label{limsupinf}
	\limsup_{t\to\infty} \|C(t)-I\| < 2,
\end{equation}
which is weaker than the premise in the \ref{sup}, and show that
\begin{equation}\label{limsupinflaw}
\tag{\em limsup-$\infty$-law} \limsup_{t\to\infty} \|C(t)-I \| < r \implies C(t)=I \ \forall t\in\R,
\end{equation}
 for $r=2$ holds, see Theorem \ref{thm:limsupinftwolaw}. A related question was raised in \cite[Remark 2.6]{SchweZwaLessthanone} for, more general, scaled versions of these laws. More precisely, it was asked whether for $a\geq0$ the following holds for some $r$,
\begin{equation}\label{scaledlimsupinflaw}
\limsup_{t\to \infty} \|C(t)-cos(at)\|<r \implies C(t)=\cos(at)\ \forall t\in\R.
\end{equation}
Let us mention that `scaled version' (where the unity element $I$ gets replaced by $\cos(at)I$) of \ref{limsup} and \ref{sup} have a different optimal constant $r=\frac{8}{3\sqrt{3}}$, see \cite{BobrowskiZeroTwo,ChojnackiOneTwo,Esterle2015}.

In the following, we show that (\ref{limsupinflaw}) holds, using techniques by J. Esterle \cite{EsterleZeroTwo}. Finally we state the corresponding result for semigroups, for which \textit{zero-one-laws} have been studied much earlier than for cosine families.
\section{A $\limsup_{t\to\infty}$- law}
In the following, for a unital normed algebra $A$, let $I$ denote the identity element.
\begin{lemma}\label{lem1}
Let $\CosineC$ be a cosine family in a unital Banach algebra. If
\begin{equation*}
\limsup_{t\to \infty} \| C(t)-I \|=0,
\end{equation*}		
then $C(t)=I$ for all $t\in\R$.
\end{lemma}
\begin{proof}
From the assumption follows that  $\lim_{t\to\infty}\|C(t)-I\|=0$. By d'Alembert's defining identity for cosine families, 
\begin{equation}
C(t+s)+C(t-s)=2C(t)C(s),
\end{equation}
for all $s,t\in\R$. Thus, letting $t\to\infty$, we derive $2I=2C(s)$ for all $s\in\R$.
\end{proof}
The following lemma is a slight extension of Esterle's Lemma 2.1 in \cite{EsterleZeroTwo}, as we also allow for $t_{0}=\infty$. The proof is completely analogous the case case $t_{0}=0$.
\begin{lemma}\label{lem2}
Let $\cosinec$ be a complex-valued cosine family and $t_{0}\in\left\{0,\infty\right\}$. Then, we have one of the following situations.
\begin{enumerate}[label=(\roman*)]
\item $\limsup_{t\to t_{0}} |c(t)-1|=\infty$,
\item $\limsup_{t\to t_{0}}|c(t)-1|=2$,
\item \label{it3}$\limsup_{t\to t_{0}}|c(t)-1|=0$.
\end{enumerate}
Moreover, in case \ref{it3}, it follows that 
\begin{equation}\label{eq22}
	c(t)=\left\{\begin{array}{ccl}1&\text{if}& t_{0}=\infty,\\\cos(at)&\text{if}& t_{0}=0,\end{array} \right. 
\end{equation}
for some $a\geq0$.
\end{lemma}
\begin{proof}
As mentioned the proof is analogous to the one in \cite[Lemma 2.1]{EsterleZeroTwo}. In case \ref{it3} and $t_{0}=\infty$, it follows by Lemma \ref{lem1} that $c(t)=1$ for all $t\in\R$.
\end{proof}
\begin{proposition}\label{prop1}
Let $\CosineC$ be a cosine family on a unital Banach algebra $A$. If $\limsup_{t\to\infty}\rho\left(C(t)-I\right)<2$, then $\rho\left(C(t)-I\right)=0$ for all $t\in\R$.
\end{proposition}
\begin{proof}
Let $\hat{A}$ denote the space of characters on $A$. For all $t\in\R$ we have that
\begin{equation}\label{eq:spectralradiuschi}
\rho\left(C(t)-I\right)=\sup_{\chi\in\hat{A}}|\chi\left(C(t)-I\right)|=\sup_{\chi\in\hat{A}} |\chi(C(t))-1|.
\end{equation}
Thus, by the assumption we get that $\limsup_{t\to\infty}|\chi(C(t))-1|<2$ for $\chi\in\hat{A}$. As  $\left(\chi\left(C(t)\right)\right)_{t\in\R}$ is a complex-valued cosine family, Lemma \ref{lem2} then implies that $\chi(C(t))=1$ for all $t\in\R$ and $\chi\in\hat{A}$. Using this in (\ref{eq:spectralradiuschi}), we deduce that $\rho(C(t)-I)=0$ for all $t\in\R$.
\end{proof}
As pointed by Esterle \cite{EsterleZeroTwo}, for a commutative unital Banach algebra $A$, for $x\in A$ with $\|x\|\leq1$ we can define
\begin{equation}
\sqrt{I-x}:=\sum_{n=0}^{\infty}(-1)^{n}\alpha_{n}x^{n},
\end{equation}
where $(-1)^{n}\alpha_{n}$, with $\alpha_{0}=1$, $\alpha_{n}=\frac{1}{n!}\frac{1}{2}(\frac{1}{2}-1)...(\frac{1}{2}-n+1)=(-1)^{n-1}\frac{1}{n2^{n-1}}\binom{2(n-1)}{n-1}$, $n>0$, are the Taylor coefficients of the function $z\to\sqrt{1-z}$ at the origin (with convergence radius equal to $1$).  
Since $(-1)^{n-1}\alpha_{n}>0$ for $n\geq1$,
\begin{equation}\label{ineq:sqrt}
\left\|I-\sqrt{I-x}\right\|\leq\sum_{n=1}^{\infty}|\alpha_{n}|\|x\|^{n}=\sum_{n=1}^{\infty}(-1)^{n-1}\alpha_{n}\|x\|^{n}=1-\sqrt{1-\|x\|}.
\end{equation}
\begin{lemma}[Esterle 2015, \cite{EsterleZeroTwo}]\label{lem3}
Let $\CosineC$ be a cosine family in a unital Banach algebra.If $\|C(2s)-I\|\leq2$ and that $\rho(C(s)-I)<1$ for some $s\in\R$, Then, 
\begin{equation*}
C(s)=\sqrt{I-\frac{I-C(2s)}{2}}, 
\end{equation*}
where the square root is defined as described above.
\end{lemma}
With the above preparatory results, the \ref{limsupinflaw} is now easy to show. The proof can be done analogously to the one in \cite[Theorem 3.2]{EsterleZeroTwo}, which in turn can be seen as an elegant refinement of the technique used in the \textit{three-lines-proof} in \cite{UlmerSeminare2012}.
\begin{Theorem}\label{thm:limsupinftwolaw}
Let $\CosineC$ be a cosine family in a unital Banach algebra $A$. Then, $\limsup_{t\to\infty} \|C(t)-I \| < 2$ implies that 
$C(t)=I$ for all $t\in\R$.
\end{Theorem}
\begin{proof}
By Proposition \ref{prop1}, we have that $\rho(C(t)-I)=0$ for $t\in\R$. Thus, we can apply Lemma \ref{lem3} and Eq.\ (\ref{ineq:sqrt}) so that for all $s\in\R$,
\begin{equation*}
\left\| I - C(s) \right\| \leq 1 - \sqrt{1-\left\|\frac{I-C(2s)}{2}\right\|}\leq1.
\end{equation*}
With $S:=\limsup_{s\to\infty}\|C(s)-I\|$, this yields that
\begin{equation*}
S\leq 1- \sqrt{1-\frac{S}{2}}\leq1,
\end{equation*}
which implies that $S=0$. Hence, Lemma \ref{lem1} concludes the assertion.
\end{proof}

\begin{rem}
After finishing this note, Esterle pointed out that, alternatively, \cite[Theorem 2.3]{Esterle2015} implies that for a bounded cosine sequence with $\rho(C(1)-I)=0$, it follows that $C(t)=\cos(at)I$ for all $t\in\R$ and some $a\in\R$. Thus, $\limsup_{t\to\infty}\|C(t)-I\|<2$ implies $C(t)=I$ for all $t\in\R$ and therefore, the use of Lemma \ref{lem3} can be omitted.

\end{rem}

\begin{rem}
It is clear that Theorem \ref{thm:limsupinftwolaw} generalizes the \ref{sup} with $r=2$. We remark that the known proofs of the \ref{sup}, see \cite{ChojnackiZeroTwo,SchweZwaZeroTwo}, which use a diagonalisation argument and the \ref{limsup}, can not be generalized to the assertion of Theorem \ref{thm:limsupinftwolaw}.
\end{rem}
\subsection{A discrete $\limsup$-law}
For discrete cosine families, or \textit{cosine sequences} $\left(C(n)\right)_{n\in\Z}$, the following was proved in \cite{SchweZwaZeroTwo} (There, it was formulated for the special case of $C:\Z\rightarrow \mathcal{B}(X)$ for a Banach space $X$. However, the proof is the same for general Banach-algebra-valued cosine families)
\begin{Theorem}[\cite{SchweZwaZeroTwo}]
 Let $\left(C(n)\right)_{n\in\Z}$ be a discrete cosine family in a unital Banach algebra. Then,
 \begin{equation*}
  \sup_{n\in\N}\|C(n)-I\|<\frac{3}{2}\ \implies \ C(n)=I \ \forall n\in\Z.
 \end{equation*}
\end{Theorem}
The proof is based on an elegant idea of Arendt, which can directly be applied to weaken the $\sup$ to $\limsup$ in the theorem.
\begin{Theorem}
 Let $\left(C(n)\right)_{n\in\Z}$ be a discrete cosine family in a unital Banach algebra. Then,
 \begin{equation*}
  \limsup_{n\to\infty}\|C(n)-I\|<\frac{3}{2}\ \implies \ C(n)=I \ \forall n\in\Z.
 \end{equation*}
\end{Theorem}
This result is optimal as can be seen by $C(n)=\cos(\frac{2n\pi}{3})$ which yields $\limsup_{n\to\infty}\|C(n)-I\|=\frac{3}{2}$, see \cite[Theorem 3.2]{SchweZwaZeroTwo}.
\section{The corresponding semigroup result}
Let us finally state the corresponding result for (discrete) semigroups in a unital normed algebra, which is a corollary of a well-known result by Wallen \cite{Wallen}.
\begin{Theorem}\label{thm:SGlimsupinflaw}
Let $\left\{T_{n}\right\}_{n\in\N}$ be a semigroup in a normed unital algebra. Then,
\begin{equation}
\limsup_{n\to\infty}\|T_{n}-I\|<1\ \implies \ T_{n}=I \ \forall n\in\N.
\end{equation}
\end{Theorem}
\begin{proof}
If $\limsup_{n\to\infty}\|T_{n}-I\|<1$, then $\liminf_{n\in\N}\frac{1}{n}\sum_{j=1}^{n}\|T_{j}-1\|<1$. By Wallen \cite{Wallen}, the assertion follows.
\end{proof}
\begin{rem}
Clearly, Theorem \ref{thm:SGlimsupinflaw} implies that for a semigroup on $[0,\infty)$, $\left(T(t)\right)_{t\geq0}$, we have that
\begin{equation}
\limsup_{t\to\infty}\|T(t)-I\|<1\ \implies \ T(t)=I \ \forall t\geq0.
\end{equation}
\end{rem}


\end{document}